\documentclass[12pt,a4paper]{article}

\usepackage{amsmath}
\usepackage{amssymb}
\usepackage{mathtools}
\usepackage{tikz}
\usepackage{graphicx}
\usepackage{amsthm}
\usepackage{fullpage}
\usepackage{authblk}

\theoremstyle{plain}
\newtheorem{thm}{Theorem}[section]

\newtheorem{lem}[thm]{Lemma}

\theoremstyle{definition}
\newtheorem{dfn}[thm]{Definition}

\DeclareMathOperator{\fit}{Fit}
\DeclareMathOperator{\aut}{Aut}
\DeclareMathOperator{\pow}{Pow}
\DeclareMathOperator{\soc}{Soc}
\DeclareMathOperator{\psl}{PSL}
\DeclareMathOperator{\pgl}{PGL}
\DeclareMathOperator{\gl}{GL}
\DeclareMathOperator{\sz}{Sz}
\DeclareMathOperator{\spl}{SL}

\usepackage[colorlinks]{hyperref}
\definecolor{lightblue}{rgb}{0.5,0.5,1.0}
\definecolor{darkred}{rgb}{0.5,0,0}
\definecolor{darkgreen}{rgb}{0,0.5,0}
\definecolor{darkblue}{rgb}{0,0,0.5}

\hypersetup{colorlinks,linkcolor=darkred,filecolor=darkgreen,urlcolor=darkred,citecolor=darkblue}

\begin{document}
\title{Classification of non-solvable groups whose power graph is a cograph
\footnote{
The research leading to these results has received funding
from the European Research Council (ERC) under the European
Union’s Horizon 2020 research and innovation programme
(EngageS: grant agreement No. 820148) and from the German Research Foundation DFG (SFB-TRR 195 ``Symbolic Tools in Mathematics and their Application''). We thank Pascal Schweitzer for the helpful comments on an earlier draft of this paper.
}
}

\author{Jendrik Brachter}
\author{Eda Kaja}
\affil{TU Darmstadt\footnote{Email addresses: brachter@mathematik.tu-darmstadt.de, kaja@mathematik.tu-darmstadt.de}}
\date{\today}
\maketitle

\begin{abstract}
Cameron, Manna and Mehatari investigated the question of which finite groups admit a power graph that is a cograph, also called \emph{power-cograph groups} (Journal of Algebra 591 (2022)). They give a classification for nilpotent groups and partial results for general groups. However, the authors point out number theoretic obstacles towards a classification. These arise when the groups are assumed to be isomorphic to $\psl_2(q)$ or $\sz(q)$ and are likely to be hard.

In this paper, we prove that these number theoretic problems are in fact the only obstacles to the classification of non-solvable power-cograph groups. Specifically, for the non-solvable case, we give a classification of power-cograph groups in terms of such groups isomorphic to $\psl_2(q)$ or $\sz(q)$. 
For the solvable case, we are able to precisely describe the structure of solvable power-cograph groups.
We obtain a complete classification of solvable power-cograph groups whose Gruenberg-Kegel graph is connected. Moreover, we reduce the case where the Gruenberg-Kegel graph is disconnected to the classification of $p$-groups admitting fixed-point-free automorphisms of prime power order, which is in general an open problem. 
\end{abstract}
\section{Introduction}

A classic reoccurring approach to investigating the structure of groups is by associating the groups with graphs and then studying their graph theoretical properties. Prominent examples include Cayley graphs, capturing the structure of groups relative to the action of a specified generating set, or orbital graphs, capturing primitivity properties of permutation groups.

A central example in the context of the present work is the \emph{Gruenberg-Kegel graph}, also called the \emph{prime graph} of a group. Rather than describing actions of groups, the prime graph $\pi(G)$ of a group $G$ is built from abstract structural aspects of the group: the vertices are the primes dividing the order of the group and two primes are joined by an edge if and only if their product is realized as the order of some group element. The prime graph was first investigated by Gruenberg and Kegel (see~\cite{gruenberg_roggenkamp_1975}, some results were later published by Williams in~\cite{williams1981} based on an unpublished manuscript by Gruenberg and Kegel) in the context of augmentation ideals of integral group rings. The structure of the prime graph has deep connections with the structure of the group itself. For instance, if $\pi(G)$ is disconnected, then the structure of $G$ is quite restricted (cf. {\cite[Theorem 14]{manna2021forbidden}}) and if $G$ is solvable, then $\pi(G)$ does not contain an induced coclique of size $3$~\cite{Lucido2002}. 

A recent, active branch of research studies graphs whose vertex sets are themselves groups.
A basic but important example in this scope is the \emph{power graph} of a group, which was first studied in the context of building graphs from semigroups \cite{KelarevQuinn}. The vertices are given by an underlying semigroup and arrows are drawn from elements to all their powers. Power graphs have been frequently studied throughout the literature, see for example~\cite{aalipour_akbari_cameron_nikandish_shaveisi_2017, abawajy_kelarev_chowdhury_2013}.

A comprehensive overview of various graphs defined on groups including many key results and open problems can be found in~\cite{cameron2021graphs}. There, Cameron points out that many commonly investigated graphs defined on groups fit into a hierarchy of subgraphs. The hierarchy includes the (directed) power graph and the enhanced power graph, the commuting graph and the deep commuting graph and the generating graph. An extension of this hierarchy into a second dimension has been studied in~\cite{arunkumar2021super}. A particular focus lies on the interactions between different graphs. For instance, it is known that all graphs from the hierarchy uniquely determine the prime graph~\cite[Theorem 2.9]{cameron2021graphs}.

The work of Cameron et al. has spawned a number of intriguing problems relating the structure theory of graphs and groups as well as problems of classifying groups in terms of restrictions on their graphs, often resulting in new characterizations of interesting known group classes. There has been ongoing research to determine the groups for which two of the graphs in the hierarchy coincide, and various conditions and results have been proven and summarized in~\cite{aalipour_akbari_cameron_nikandish_shaveisi_2017, cameron2021graphs}.  While investigating groups that are unrecognizable by their prime graph, Cameron and Maslova recently gave a classification of non-solvable \emph{EPPO-groups}~\cite{cameron2021criterion}. EPPO-groups are groups in which every element has prime power order. The study of such groups was initiated by Higman in~\cite{higmanEPPO}, and later extended and generalized by numerous authors, see for example~\cite{buturlakin_2017,shumyatsky_2020}, including the famous work of Suzuki~\cite{suzukiCentralizers,suzukiDoublyTransitive} on doubly transitive groups.

In the present work we are concerned with generalizing the classification of EPPO-groups by extending it to a classification of groups whose power graph is a cograph (in the following called \emph{power-cograph groups}). The class of power-cograph groups indeed contains all EPPO-groups by an argument of Cameron \cite{cameron2021graphs}.
A \emph{cograph} is a graph that does not contain a path on four vertices as an induced subgraph. Cographs are frequently studied in structural graph theory and have been rediscovered numerous times, see for example~\cite{jung_1978, seinsche_1974, sumner_1974, corneil_1981}. An important theme in structural graph theory is the study of hereditary graph classes which can be defined via forbidden induced minors~\cite{graph_classes}. After cluster graphs (defined by forbidding induced paths on three vertices), cographs form a basic and natural class of graphs to consider, as they are defined by a single forbidden induced subgraph.

The problem of classifying power-cograph groups goes back to~\cite{manna2021forbidden}, where the authors investigate general forbidden substructures in power graphs. They also point out that classifying such groups when isomorphic to $\psl_2(q)$ is equivalent to a number theoretic problem that is likely to be hard. Further investigations in \cite{cameron2021finite} show that the classification of simple power-cograph groups completely reduces to similar number theoretic problems for groups isomorphic to $\psl_2(q)$ or $\sz(q)$ (see~\cite[Theorem 1.3]{cameron2021finite} and the discussion after). Nilpotent power-cograph groups are fully classified and partial results are available for the solvable case in \cite{cameron2021finite}.
To this end, the best we can expect is a classification of non-solvable power-cograph groups, relative to the evident number theoretic obstacles, which is what we provide in this paper.
Given a solution to these problems, our results reveal a complete classification of non-solvable power-cograph groups.  In particular, for the non-solvable case we prove the following theorem.

\begin{thm}[Non-solvable power-cograph groups]\label{mainthm2}
Let $G$ be a finite group that is not solvable. The power graph of $G$ is a cograph if and only if one of the following holds:
	\begin{enumerate}
		\item $G$ is a simple power-cograph group isomorphic to one of $\psl_2(q)$, $\sz(q)$ or $\psl_3(4)$. The admissible values of $q$ are precisely characterized in \cite{cameron2021finite}, see Theorem~\ref{thm:finite_simple_po-cographs} below.
		\item $G$ is one of $\pgl(2,5)$, $\pgl(2,7)$, $\pgl(2,9)$ or $M_{10}$.
		\item $T:=G/\soc(G) \cong\psl_2(2^n)$ with $n\geq 2$ and $T$ is a power-cograph group. Furthermore, $\soc(G)=\mathcal{O}_2(G)$ and each minimal normal subgroup of $G$ is isomorphic to the natural module over the group ring $\mathbb{F}_{2^n}[\textsc{SL}_2(2^n)]$ as a $T$-module.
		\item $T:=G/\soc(G) \cong\sz(2^{2e+1})$ with $e\geq 2$ and $T$ is a power-cograph group.  Furthermore, $\soc(G)=\mathcal{O}_2(G)$ and each minimal normal subgroup of $G$ is isomorphic to the natural module over the group ring $\mathbb{F}_{2^{2e+1}}[\sz(2^{2e+1})]$ as a $T$-module.
	\end{enumerate}
\end{thm}

Moreover, our techniques also yield new insights for the solvable case.  We 
obtain a complete classification of solvable power-cograph groups with connected prime graph. In the disconnected case one problem remains, and it already arises in the case of solvable EPPO-groups. In general, much is known about the structure of such groups since they are always Frobenius or 2-Frobenius groups. However, to date there is no complete classification of solvable Frobenius groups available. In particular, the question of which $p$-groups admit fixed-point-free automorphisms is still open. However, this is the only remaining obstacle and we precisely describe the structure of solvable power-cograph groups in terms of the following theorem.

\begin{thm}[Solvable power-cograph groups]\label{mainthm1}
Let $G$ be a finite solvable group. The power graph of $G$ is a cograph if and only if one of the following holds:
	\begin{enumerate}
		\item $G$ has prime power order.
		\item $G= (C_p\rtimes C_{q^n})\times C_q$, and if $n>0$ then $C_{q^n}$ acts fixed-point-freely on the $C_p$-part of $\fit(G)$.
		\item $G = C_{pq}\rtimes C_{r^n}$, and if $n>0$ then $C_{r^n}$ acts fixed-point-freely on $C_{pq}$. Thus $r^n$ divides $q-1$ and $p-1$.
		\item $G = C_{rs}\rtimes C_{pq}$, where $pq$ divides $r-1$ and $s-1$ and $C_{pq}$ acts fixed-point-freely on $C_{rs}$.
		\item $F:=\fit(G)$ is an $r$-group for some prime $r$ and $H:=G/F$ is in one of three cases:
		\begin{enumerate}
			\item $H$ is a generalized quaternion $2$-group acting without fixed points on $F$.
			\item $F$ admits a group of fixed-point-free automorphisms $P\cong C_{p^k}$ for a prime $p\neq r$ and $H$ is a subgroup of $P\rtimes C_{r^m}$. The
			action of $C_{r^m}$ on $P$ is fixed-point-free if $m>0$.
			\item $F$ admits a group of fixed-point-free automorphisms $C\cong C_{pq}$ for
			primes $p\neq q$ with $p,q\neq r$. Here, $H$ is a subgroup of $C\rtimes C_{r^m}$. The action of $C_{r^m}$ on $C$ is fixed-point-free if $m>0$.
		\end{enumerate}			 
	\end{enumerate}
\end{thm}
Note that the groups in Case 5a have been fully classified, see~\cite{mayr_1999}: $F$ is necessarily abelian and $F$ admits a fixed-point-free action by a generalized quaternion group of order $2^n$ if and only if $2^n$ divides $p^k-1$ for every cyclic direct factor of $F$ of order $p^k$~\cite[Proposition 6.24]{mayr_1999}. More generally,~\cite{mayr_1999} contains a classification of all groups in Case 5 when $F$ is assumed to be abelian. In the non-abelian case there is an extensive body of research concerned with bounding the Fitting-length of $G$ as a function of the order of a fixed-point-free automorphism of prime order~\cite{higman_1957, ward_1969, thompson_1959}, but no general classification is available.

\paragraph{Techniques} Regarding our techniques, we first exactly characterize power-cograph groups in terms of properties of centralizers of their group elements. This relates back to the classification of \emph{CN-groups} (groups with nilpotent element-centralizers, see \cite{Feit1960FiniteGI}) and \emph{CIT-groups} (groups in which involution centralizers are $2$-groups, see \cite{suzukiCentralizers,suzukiDoublyTransitive}). The latter includes the discovery of the famous Suzuki-groups. In these works, certain restrictions on element centralizers are shown to have strong influence on the overall structure of the ambient group.
In a similar fashion, the centralizer restrictions we obtain allow for non-nilpotent centralizers but are otherwise severely limiting. We are able to employ results from~\cite{martineau_1972} and~\cite{stewart_1973} (building on work of Suzuki in~\cite{suzukiCentralizers,suzukiDoublyTransitive}) in dealing with extensions of $2$-groups by non-abelian simple groups. We point out that same results are also employed in the classification of non-solvable EPPO-groups in~\cite{cameron2021criterion}.

\paragraph{Outline of the paper} We give some preliminary information about notation and the relevant theory of finite groups in Section~\ref{section:preliminaries}. Then we discuss power-cograph groups in Section~\ref{section:po-co groups} and we go into detail about centralizers in Section~\ref{section:centralizers}. For the classification, we consider power-cograph groups in terms of the solvability of the group. We discuss solvable groups in Section~\ref{section:solvable} and non-solvable ones in Section~\ref{section:non-solvable}. The proofs of the main theorems can be found in Section~\ref{section:summary}.

\section{Preliminaries}\label{section:preliminaries}
In the following we collect some necessary group theoretic basics. Readers familiar with finite group theory might want to skip this section. As a general reference for basics on group theory, we refer to~\cite{gorenstein_1968} and \cite{passman_1968}.

\paragraph{Notation}
Given a set of primes $\pi$, we say a \emph{$\pi$-element} in a group $G$ is an element whose order is only divisible by primes in $\pi$. Similarly a \emph{$\pi$-group} is a group that entirely consists of $\pi$-elements.
We say that an automorphism of a group $H$ is \emph{fixed-point-free} if it only fixes the identity element in $H$. Given a normal subgroup $N$ of a group $G$, then $G/N$ acts \emph{fixed-point-freely} on $N$ if the induced action by conjugation is fixed-point-free for every non-trivial element in $G/N$. General extensions of $H$ by $N$, i.e., groups with normal subgroup isomorphic to $N$ and corresponding quotient isomorphic to $H$, are denoted by $N.H$.
We use $U\lesssim G$ to denote that $U$ is isomorphic to a subgroup of $G$.
\paragraph{Group structure}

The following theorems fundamental to group structure theory are invoked throughout this paper.

\begin{thm}[Schur-Zassenhaus, see~\cite{gorenstein_1968}]
	Let $G$ be a finite group admitting a normal subgroup $N\trianglelefteq G$, such that $|N|$ and $|G/N|$ are coprime. Then $N$ has a complement in $G$, that is, $G$ is isomorphic to the semidirect product $N\rtimes G/N$.
\end{thm}

\begin{thm}[Burnside's $p^aq^b$-Theorem,~\cite{burnside_1904}]
If a finite group $G$ is of order $p^a q^b$ with primes $p$ and $q$, then $G$ is solvable.
\end{thm}
As a consequence, the order of every finite non-abelian simple group is divisible by at least three distinct primes.

\paragraph{Fitting-subgroup and socle}
Recall that the \emph{Fitting-subgroup} $\fit(G)$ of a group $G$ is the (unique) largest
normal nilpotent subgroup in $G$. Furthermore, let $\fit_2(G)$ denote the subgroup of $G$ that satisfies $\fit_2(G) / \fit(G) = \fit(G/\fit(G))$.
Given a prime $p$, we denote the (unique) largest normal $p$-subgroup of a group $G$ by $\mathcal{O}_p(G)$ (also called the \emph{$p$-radical}). In case that $G$ is a finite group, it holds that
\[	
	\fit(G)=\prod_{p\mid |G|}\mathcal{O}_p(G)
\] is a direct product over all $p$-radicals in $G$. If $G$ is furthermore solvable, then $\fit(G)$ is self-centralizing, i.e.,~$C_G(\fit(G))=Z(\fit(G))$, and then $G/\fit(G)$ is isomorphic to a subgroup of $\aut(\fit(G))$.

The \emph{socle} of $G$, denoted $\soc(G)$, is the subgroup generated by all minimal normal subgroups of $G$. The socle of a finite group is a direct product of simple groups and if $\fit(G)$ is trivial, then all direct factors of $\soc(G)$ are non-abelian simple.

\paragraph{Frobenius groups}
Let $H$ be a proper nontrivial subgroup of $G$ such that $H \cap H^x = 1$ for all $x \in G\setminus H$. Then $K=G\setminus \bigcup_{x\in G} (H\setminus 1)^x$ is a normal subgroup of $G$ such that $G= K \rtimes H$. In this case, we say that $G$ is a \textit{Frobenius} group, with \textit{Frobenius complement} $H$ and \textit{Frobenius kernel} $K$.  

Let $G$ be a finite Frobenius group with kernel $K$ and complement $H$. Then the following hold (for proofs see for example~\cite{passman_1968, thompson_1959, wilson_2009, zassenhaus_1935}):
	\begin{enumerate}
	    \item $K$ is nilpotent and in fact~$K=\fit(G)$.
	    \item For all $k\in K$, we have $C_H(k)=1$.
	    \item If $H$ has even order, then $K$ is abelian.  
        \item If $G$ is solvable, then the Sylow~$p$-subgroups of $H$ are cyclic if $p > 2$ and cyclic or generalized quaternion if $p=2$.
        \item If a Frobenius complement $H$ is not solvable, then it has a normal subgroup of index $1$ or $2$ that is the direct product of $\mathrm{SL}(2,5)$ and a metacyclic group of order coprime to $30$.
	\end{enumerate}

A group $G$ is said to be a \textit{$2$-Frobenius group} if there exist two normal subgroups $F$
and $L$ of $G$ such that $G / F$ is a Frobenius group with kernel $L / F$ and $L$ is a Frobenius group with kernel $F$. Finally, let us point out the following well known fact about $2$-Frobenius groups.

\begin{lem}[{\cite[Theorem 2]{chen_1995}}]
	All $2$-Frobenius groups are solvable.
\end{lem}

\paragraph{Outer automorphisms of $\psl_2(q)$ and $\sz(q)$}
We just give a brief summary here, for more details see~\cite{wilson_2009}.
The outer automorphisms of the classical groups generally come in three different flavors: diagonal, field and graph automorphisms.

For $\psl_2(q)$, the diagonal automorphism corresponds the the action of $\textsc{PGL}_2(q)$ on $\psl_2(q)$. It is induced by conjugation with 
$$\begin{pmatrix}
1 & 0\\
0 & \epsilon
\end{pmatrix},\text{ where $\epsilon$ is a non-square in $\mathbb{F}_q^*$.}$$  In particular, diagonal automorphisms only exist for odd $q$. In case of $\psl_2(q)$, the graph automorphisms appear as inner automorphisms.  

The field automorphisms are induced by automorphisms of the underlying field $\mathbb{F}_q$ of order $q=p^f$, whose automorphism group is cyclic of order $f$, generated by the Frobenius automorphism $x \to x^p$. The induced group of field automorphisms acts on $\psl_2(q)$ by mapping each matrix entry to its $p$-th power (for any matrix representing the respective element of $\psl_2(q)$).

In the case of $\sz(q)$, the simple Suzuki groups only have field automorphisms as outer automorphisms, acting on the natural representation of $\sz(q)$ in $\gl_4(q)$.

\paragraph{The Gruenberg-Kegel graph (or prime graph)}
The \textit{Gruenberg-Kegel graph} or the \textit{prime graph} of $G$, denoted $\pi(G)$, has vertex set the prime divisors of $|G|$, and two primes $p$ and $q$ are adjacent if and only if there exists an element of order $pq$ in $G$.
We refer to the following result as Lucido's Three Prime Lemma:
\begin{lem}[{\cite[Proposition 1]{Lucido1999}}]
Let $G$ be a finite solvable group. If $p, q, r$ are three
different primes dividing $|G|$, then $G$ contains an element whose order is the product of two of these primes.
\end{lem}

In~\cite{williams1981}, Williams classified the number of connected components of the prime graph for the simple groups of Lie type over fields of odd characteristic, alternating groups, and the $26$ sporadic simple groups. The case of simple groups of Lie type in even characteristic has been treated in~\cite{kondratev1989, iiyori_yamaki_1993}.
The following result for prime graphs of solvable groups is due to Gruenberg and Kegel (published by Williams in~\cite{williams1981} and slightly modified in~\cite{manna2021forbidden}):

\begin{lem}[{\cite[Theorem 3.4]{manna2021forbidden}}]\label{lem:disconnected_pi}
The prime graph of $G$ is disconnected only if $G$ satisfies one of the following:
\begin{enumerate}
    \item $G$ is Frobenius;
    \item $G$ is $2$-Frobenius;
    \item $|G|$ is even and $G$ is an extension of a nilpotent $\pi_1$-group by a simple group by a $\pi_1$-group. Here $\pi_1$ denotes the component of $2$ in $\pi(G)$.
\end{enumerate}
\end{lem}

\section{Power graphs and power-cograph groups}\label{section:po-co groups}
Throughout this work we only consider finite groups. The following section collects basic definitions and some known results on power-cograph groups.
\begin{dfn}
	The \textit{power graph} $\pow(G)$ of a group $G$ is the undirected graph defined over the vertex set $G$ with an edge joining $g,h\in G$ if and only if $g^n=h$ or $h^n=g$ holds for some $n\in\mathbb{N}$.
\end{dfn}

\begin{dfn}
	A graph is a \emph{cograph} if it does not contain $P_4$ (i.e., a path with four vertices) as an induced subgraph. We say a group $G$ is a \emph{power-cograph group} if $\pow(G)$ is a cograph.
\end{dfn}

The power graph of a subgroup $H\leq G$ is an induced subgraph of $\pow(G)$. In particular, the class of power-cograph groups is closed under taking subgroups. 

\begin{lem}
Let $G$ be an arbitrary finite group.
\begin{enumerate}
	\item Let $g,h\in G$ such that $\langle g\rangle=\langle h\rangle$ holds, then $g$ and $h$ form twins in $\pow(G)$, i.e., they are joined by an edge and have identical sets of neighbors.
	\item If $G$ is a cyclic $p$-group, then $\pow(G)$ is a clique.
\end{enumerate}
\end{lem}

In special cases the power-cograph property also propagates to quotients.
\begin{lem}\label{lem:lifting_P4}
	Let $G$ be a finite group, $g_1,\dots,g_4\in G$ and $N\trianglelefteq G$.
	Assume that $(g_1N,\dots,g_4N)$ induces a $P_4$ in $\pow(G/N)$ and that (w.l.o.g.) $g_2N$ is a power of both $g_1N$ and $g_3N$.
	If $\gcd(|N|,|C_{G/N}(g_2N)|)=1$, then $(g_1,\dots,g_4)$ induces a~$P_4$ in $G$.
\end{lem}
\begin{proof}
	We assume that $\gcd(|N|,|C_{G/N}(g_2N)|)=1$ holds. By the Schur-Zassenhaus Theorem, $G$ contains a subgroup that is a split extension $N\rtimes C_{G/N}(g_2N)$. In particular,
	there exists a subgroup of $G$ that is isomorphic to $C_{G/N}(g_2N)$ and the latter contains all the $g_iN$.
\end{proof}

Finite nilpotent power-cograph groups are characterized in~\cite{manna2021forbidden} as follows:

\begin{thm}[{\cite[Theorem 12]{manna2021forbidden}}]\label{thm:nilpotent_poco}
Let $G$ be a finite nilpotent group. Then $Pow(G)$ is a cograph if and only if either $|G|$ is a prime power, or $G$ is cyclic of order $pq$ for distinct
primes $p$ and $q$.
\end{thm}

Recent work of Cameron, Manna and Mehatari~\cite{cameron2021finite} investigates several graph classes in an attempt to determine groups whose power graph is a cograph. Their main result on the class of finite simple groups states the following:

\begin{thm}[{\cite[Theorem 3.2]{cameron2021finite}}]\label{thm:finite_simple_po-cographs}
Let $G$ be a non-abelian finite simple group. Then $G$ is a power-cograph if and only if one of the following holds:
	\begin{enumerate}
		\item $G=\psl_2(q)$, where $q$ is an odd prime power with $q\geq 5$, and each of $(q-1)/2$ and $(q+1)/2$ is either a prime power or the product of two distinct primes;
		\item $G=\psl_2(q)$, where $q$ is a power of $2$ with $q\geq 4$, and each of $q-1$ and $q+1$ is either a prime power or the product of two distinct primes;
		\item $G=\sz(q)$, where $q=2^{2e+1}$ for $e\geq 2$, and each of $q-1$, $q+\sqrt{2q}+1$ and $q-\sqrt{2q}+1$ is either a prime power or the product of two distinct primes;
		\item $G=\psl_3(4)$.
	\end{enumerate}
\end{thm}

The discussion that follows this theorem in~\cite{cameron2021finite} acknowledges the difficulties of solving the number-theoretical problems, but they provide some small values for which the theorem surely holds. However, the problem of determining if there are infinitely many such groups in the first three cases remains open.

\section{Centralizers in power-cograph groups}\label{section:centralizers}
Centralizers in power-cograph groups turn out to be quite restricted, which forms the basis of our analysis. 

\begin{lem}\label{lem:basic_obstructions_for_p_elts}
	Consider commuting group elements $x,w\in G$ with $|x|=p$ and $|w|=q$ for
	two distinct primes $p$ and $q$. If $Pow(G)$ is a cograph, then all $q$-elements in $C_G(x)$ are contained in
	$\langle w\rangle$. Furthermore, $x$ is not a $p$-th power in $G$.
\end{lem}
\begin{proof}
	Assume that there is another element of order $q$ in $C_G(x)$, $z$ say, such that $z\not\in \langle w\rangle$. Then 
	\[
		z^p\longleftarrow xz\longrightarrow x^q\longleftarrow xw
	\]is an induced $P_4$ in $Pow(G)$ (we draw directions for clarity), contradicting the assumption that $Pow(G)$ is a cograph.
	Thus, $\langle w\rangle$ is the unique subgroup of order $q$ in $C_G(x)$. Again using the 
	assumption that $Pow(G)$ is a cograph, we note that, by the classification of nilpotent power-cograph groups (Theorem~\ref{thm:nilpotent_poco}) and due to subgroup-closedness of power-cograph groups, there are no elements of order $q^2$ in $C_G(x)$. 
	
	Similarly, if $x$ has a $p$-th root, say $g^p=x$ for a suitable $g\in G$, then
	\[
		w^p\longleftarrow xw\longrightarrow x^q\longleftarrow g^{q}
	\]	is an induced $P_4$ in $\pow(G)$.
\end{proof}

\begin{lem}\label{lem:centralizers_nonisolated_elements}
	Consider commuting group elements $x,w\in G$ with $|x|=p$ and $|w|=q$ for
	primes $q < p$. Assume $Pow(G)$ is a cograph. Then
	\begin{enumerate}
		\item $C_G(x)=\langle w,x\rangle\cong C_p\times C_q$, and
		\item $C_G(w)=\langle x\rangle\rtimes Q$,
		where $\langle x\rangle\cong C_p$ and $Q\cong\langle w\rangle\times C_{q^n}$ for some $n\in\mathbb{N}_0$. If $n>0$ then $Q/\langle w\rangle$ acts without fixed points on $\langle x\rangle$, so $q^n$ divides $p-1$. 
	\end{enumerate}
\end{lem}
\begin{proof}
	\begin{enumerate}
		\item By Lemma \ref{lem:basic_obstructions_for_p_elts}, $\langle w\rangle$ is normal in $C_G(x)$.
		Since $q<p$, there is a unique normal Sylow $p$-subgroup $P$ of $C_G(x)$ by the Sylow Theorems. Hence, $C_G(x)=\langle w\rangle\times P$ is nilpotent and thus $P=\langle x\rangle\cong C_p$ via Theorem~\ref{thm:nilpotent_poco}.
		\item As in the first case, $\langle x\rangle\trianglelefteq C_G(w)$ and then
		$C_G(w)=\langle x\rangle\rtimes Q$ for a $q$-group $Q$ by the Schur-Zassenhaus Theorem.
		Again using Lemma \ref{lem:basic_obstructions_for_p_elts}, we note that, up to powers, $w$ is the unique $q$-element
		in $C_G(x)$. So $C_Q(x)=\langle w\rangle$ and thus $Q/\langle w\rangle\leq \aut(\langle x\rangle)\cong C_{p-1}$. By definition, $w\in Z(Q)$ which implies that $Q/Z(Q)$ is cyclic, hence $Q$ is abelian. The lemma further gives us that $w$ is not a $q$-th power and thus $Q\lesssim\langle w\rangle\times C_{q^n}$.\qedhere
	\end{enumerate}
\end{proof}

\begin{lem}\label{lem:sylows_nonisolated_elements}
	If $Pow(G)$ is a cograph and $\{p,q\}$ is an edge of $\pi(G)$ with $q<p$, then the Sylow $p$-subgroups of $G$ are isomorphic to $C_p$ and Sylow $q$-subgroups of $G$ are of the form
	$C_q\times C_{q^n}$ with $n\in\mathbb{N}_0$ or have cyclic center.
\end{lem}
\begin{proof}
	By the previous lemma, there exists a $p$-element $x\in G$ with $C_G(x)\cong C_p\times C_q$. Consider a Sylow $p$-subgroup $P$ of $G$ containing $x$. Then $1\neq Z(P)\leq C_G(x)$, so $Z(P)=\langle x\rangle$.
	But then even $P\leq C_G(x)$ holds, thus $P=\langle x\rangle$.
	
	Now consider $w\in G$ with $|w|=q$ and assume $p$ divides $|C_G(w)|$. In this case the previous lemma still implies $C_G(w)\cong C_p\rtimes Q$, where $Q$ is isomorphic to $\langle w\rangle\times C_{q^n}$ for some $n\geq0$.
	Let $H$ be a Sylow $q$-subgroup of $G$ containing $Q$. If $w\in Z(H)$, then $H\lesssim (C_q\times C_{q^n})$. If $w\not\in Z(H)$, then $Z(H)\lesssim C_{q^n}$.
\end{proof}

\begin{lem}\label{lem:semidirect_nontrivial_prime_graph}
	Consider a semidirect product $G\cong Q\rtimes C_p$ where $Q$ is a $q$-group for some prime $q\neq p$. If $Pow(G)$ is a cograph and $\pi(G)$ is not the null graph, then $G\cong C_q\times C_p$.
\end{lem}
\begin{proof}
	The Sylow $p$-subgroups of $G$ are isomorphic to $C_p$ and in particular, two $p$-elements
	in $G$ always generate conjugate cyclic subgroups. Let $x$ be an element of $G$ of order $p$. Since $\pi(G)$ is non-trivial, there must be some element $w\in Q$ of order $q$ with $w\in C_G(x)$. Assume there is some
	$q$-element $y\in C_G(w)\setminus\langle w\rangle$. Then $yx\in C_G(w)$ and due to the semidirect structure, $p$ divides $|yx|$. But up to powers, $x$ is the only $p$-element in $C_G(w)$, so $(yx)^d=x$ for an appropriate integer $d$. In particular it follows that
	$yx\in C_G(x)$ holds. On the other hand, by Lemma~\ref{lem:basic_obstructions_for_p_elts}, $w$ is the only $q$-element in $C_G(x)$, so
	$x$ does not commute with $y\in G\setminus\langle w\rangle$ and thus, $x$ does not commute
	with $yx$, a contradiction.	

	Hence, $C_G(w)=\langle w\rangle\times\langle x\rangle$ and, as before, this implies $Q=\langle w\rangle$ due to $Z(Q)\neq 1$.
\end{proof}

\begin{lem}\label{lem:pi_star}
	Let $G$ be a group and $Pow(G)$ a cograph. Consider a connected component $K$
	of $\pi(G)$. Then $K$ is a star with center $q$, where $q$ is the smallest prime in $K$. 
\end{lem}
\begin{proof}
	Let $q<p$ be primes that form an edge in $\pi(G)$. Then $C_G(x)\cong C_p\times C_q$
	for some element $x$ of order $p$ and furthermore, $\langle x\rangle$ is a Sylow $p$-subgroup of $G$. Since Sylow $p$-subgroups are all conjugate in $G$, it follows that
	all primes adjacent to $p$ in $\pi(G)$ already appear in $|C_G(x)|$ and then the lemma follows by repeating the argument for all primes in $K\setminus\{q\}$.
\end{proof}

We can actually give a full characterization of power-cograph groups in terms of centralizer properties. We want to point out that \cite[Theorem 3.1]{cameron2021finite} contains a characterization of possible induced $P_4$'s in a finite group in terms of element orders and this is essentially the same argument we repeat in the following proof. We combine this with our restrictions on centralizers to turn it into a characterization in terms of centralizer structure.

\begin{lem}\label{lem:characterization_via_centralizers}
	$G$ is a power-cograph group if and only if for each prime $p$ and for each element $g\in G$ of order $p$ the following hold:
	\begin{enumerate}
		\item if a prime $q<p$ divides $|C_G(g)|$, then $C_G(g)\cong C_q\times C_p$.
		\item if a prime $q>p$ divides $|C_G(g)|$, then there exists a normal subgroup $N\trianglelefteq C_G(g)$, $N\cong C_q$, and $C_G(g)/N\cong\langle g\rangle\times C_{p^n}$ with $n\in\mathbb{N}_0$. 
	\end{enumerate}
	In particular, if $C_G(g)$ is not a $p$-group, then $g$ is not a $p$-th power in $G$.
\end{lem}
\begin{proof}
	If $G$ is a power-cograph group then centralizers of prime order elements behave as claimed according to Lemma \ref{lem:centralizers_nonisolated_elements}. For the other direction, assume that $\pow(G)$ is not a cograph. Then there exists an induced $P_4$ in $\pow(G)$, given by $(g_1,g_2,g_3,g_4)$, say. Note that any two consecutive vertices cannot be of the same order, since elements generating the same cyclic group form a twin class in $\pow(G)$. So without loss of generality we may assume that $|g_2|>|g_3|$ holds. If $|g_2|$ is a power of some prime $q$, then either Condition 2 is violated or all of $g_1,\dots{},g_3$ are $q$-elements in $\langle g_2\rangle$, contradicting the assumption that the induced subgraph is a path. Thus, we may further assume that $|g_2|=pq$ with primes $p\neq q$. Then $g_1$ and $g_3$ need to be of order $p$ and $q$, respectively. In particular, $|g_4|>|g_3|$ holds. Say $|g_3|=q$. If $|g_4|=pq$, then $g_3$ commutes with two different elements of order $p$, contradicting our assumptions in any case. Otherwise $|g_4|$ is a power of $q$, but then $g$ is a $q$-th power, again contradicting our assumptions.
\end{proof}

\section{Solvable groups}\label{section:solvable}

We focus on the case of solvable power-cograph groups first. Combining Lucido's Three Prime Lemma with our observations on centralizers in power-cograph-groups yields effective restrictions on the prime graph.
\begin{lem}\label{lem:form_of_solvable_pi}
	If $G$ is a solvable power-cograph group, then $\pi(G)$ is an induced subgraph of
	either $2K_2$, the graph consisting of two disjoint edges, or $P_3$, a path on 3 vertices.
\end{lem}
\begin{proof}
	By Lucido's Three Prime Lemma, $\pi(G)$ has at most two connected components and if it does admit two components, then they can have at most two vertices each, since they are star graphs by Lemma \ref{lem:pi_star} and there cannot be an induced coclique of size 3 by Lucido's Three Prime Lemma.
	
	If $\pi(G)$ is connected, then Lemma \ref{lem:pi_star} implies that $\pi(G)$ is a star with some central vertex $p$ and $\pi(G)-\{p\}$ can not have more than two vertices.
\end{proof}

\begin{lem}\label{lem:classification_connected_pi}
	Let $G$ be solvable and assume $\pow(G)$ is a cograph. If $\pi(G)$ is connected then
	one of the following holds:
	\begin{enumerate}
		\item $G$ is a group of prime power order.
		\item $\pi(G)$ consists of a single edge $\{q,p\}$ with $q<p$.
		Then $G\cong (C_p\rtimes C_{q^n})\times C_q$, and if $n>0$, then $q^n$ divides $p-1$ and $C_{q^n}$ acts fixed-point-freely on the $C_p$-part of $\fit(G)$.
\end{enumerate}	 
	Conversely, all such groups are power-cograph groups.
\end{lem}
\begin{proof}
	By the previous lemma, $\pi(G)$ is a star with center, say, $q$. Also all other primes dividing $|G|$ are larger than $q$ and divide $|G|$ exactly once. Assume from now on that $G$ is not a $q$-group.
	
	\paragraph{Case 1: $\pi(G)=\{\{q,p\}\}$.} 	
	Set $F:=\fit(G)$. Define $F_2\leq G$ implicitly via $F_2/F:=\fit(G/F)$. If $F$ is a $p$-group then $F\cong C_p$ is the unique subgroup of order $p$ in $G$ and self-centralizing, contradicting the fact that $\pi(G)$ is connected. Thus $q$ divides $|F|$. Moreover the Fitting-group of a finite group is the direct product over all maximal normal prime power subgroups. Thus, if we assume $F$ to be a $q$-group, we obtain that $\gcd(|F|,|F_2|)=1$. By the Schur-Zassenhaus Theorem this implies that $F_2\cong F\rtimes C_p$ is a semidirect product. Since $G$ is solvable, it holds that $C_{G/F}(F_2/F)\leq F_2/F$ and then for $\pi(G)$ to be non-trivial, $\pi(F_2)$ must be non-trivial. But then $F_2\cong C_p\times C_q$ by Lemma \ref{lem:semidirect_nontrivial_prime_graph}, a contradiction.
	
	In conclusion, it must hold that $F$ is a nilpotent group whose order is divisible by $pq$, so $F\cong C_q\times C_p$. Then it holds that
	$G\lesssim (C_p\rtimes Q)\times C_q$, where $Q\leq \aut(C_p)$ is a $q$-subgroup
	of the cyclic group $\aut(C_p)$ acting fixed-point-freely on $C_p$. The fact that $F$ has a complement can be seen as follows. Let $w\in F$ be an element of order $q$. Then $w$ is centralized by each Sylow $q$-subgroup of $G$ for $q$ does not divide $|\aut(\langle w\rangle)|$. By Lemma \ref{lem:centralizers_nonisolated_elements}, the Sylow $q$-subgroups of $C_G(w)$ are isomorphic to $\langle w\rangle\times C_{q^n}$ for some $n$.

	\paragraph{Case 2: $\pi(G)=\{\{q,p\},\{q,r\}\}$.} As before, $F$ cannot be of order $p$ or $r$, since then the respective prime would have to be isolated in $\pi(G)$ and $F$ cannot be a $q$-group for the exact same reason we gave in Case 1 (note that the Fitting series of $G$ does not contain a quotient of order $pr$). Then we have $F\cong C_q\times C_p$ (or $F\cong C_q\times C_r$ which is just a matter of renaming) and $G/F\leq\aut(C_p\times C_q)$. Now there is a unique subgroup of order $q$ in $F$ and thus elements of order $r$ in $G$ have to act fixed-point-freely on this subgroup (otherwise its centralizer would be divisible by $q$, $p$ and $r$). But then $r$ must divide $q-1$, contradicting the fact that $q$ is the minimal prime in its component of $\pi(G)$. Thus, this case does in fact not occur.

	\paragraph{The converse statement:} Recall that groups of prime power order are always power-cograph groups by Theorem~\ref{thm:nilpotent_poco}. In the second case, the cograph property can be deduced from Lemma~\ref{lem:characterization_via_centralizers}.
\end{proof}

In particular, if $G$ is a solvable power-cograph group and $\fit(G)$ is of prime power order, then $\pi(G)$ is disconnected. Here we should point out that in \cite[Example 4.6.]{cameron2021finite} the authors claim that the non-trivial semidirect product $H_3\rtimes C_2$ is an example of a solvable power-cograph group whose socle is a $3$-group (here, $H_3$ denotes the Heisenberg group of order 27). Since its prime graph is connected, our classification actually shows that the group is not a power-cograph group. Since the group contains centralizers isomorphic to $C_3\times S_3$, the example already disagrees with Lemma \ref{lem:basic_obstructions_for_p_elts}.

\begin{lem}\label{lem: classification_disconnected_pi}
	Let $G$ be solvable and assume $\pow(G)$ is a cograph. If $\pi(G)$ is disconnected, then
	one of the following holds:
	\begin{enumerate}
		\item the graph $\pi(G)$ is the empty graph on two vertices. Here $\fit(G)$ 
		is a $p$-group and $G/\fit(G)$ is a generalized quaternion $2$-group or a subgroup of $C_{q^k}\rtimes C_{p^m}$ for a prime $q\neq p$ and if $m>0$ then $C_{p^m}$ acts fixed-point-freely on $C_{q^k}$. In any case, $\fit_2(G)/\fit(G)$ acts fixed-point-freely on $\fit(G)$.
		\item the graph  $\pi(G)$ consists of an edge $\{p,q\}$ together with a third, isolated vertex $r$. Then $G$ is in one of two cases: 
		\begin{enumerate}
			\item $G=C_{pq}\rtimes C_{r^n}$ with $n \geq 1$ and $C_{r^n}$ acts fixed-point-freely on $C_{pq}$ (in particular, $r^n$ divides $q-1$ and $p-1$).
			\item $\fit(G)$ is an $r$-group and $\fit_2(G)\cong C_{pq}$ acts fixed-point-freely on $\fit(G)$. Moreover, $G/ \fit_2(G)$ is either trivial or $C_{r^n}$ acting as a fixed-point-free subgroup of $\aut(C_{pq})$.
		\end{enumerate}
		\item the graph  $\pi(G)$ consists of two disjoint edges $\{p,q\}$ and $\{r,s\}$ where $pq$ divides $r-1$ and $s-1$. Then $G$ is of the form $C_{rs}\rtimes C_{pq}$ with $C_{pq}$ acting fixed-point-freely on $C_{rs}$. 
	\end{enumerate}

		Conversely, all groups described above are power-cograph groups.
\end{lem}
\begin{proof}
	\begin{enumerate}
	\item If $\pi(G)$ is the empty graph, then $\fit(G)$ has prime power order.
Moreover, $\fit_2(G)/\fit(G)$ is of prime power order coprime to $|\fit(G)|$.
 By~\cite[10.5.6]{robinson_1996}, the group $\fit_2(G)/\fit(G)$ is cyclic or a generalized quaternion $2$-group. In the first case, $G=\fit_2(G)$ or otherwise $G$ is a $2$-Frobenius group where $G/\fit_2(G)$ is a cyclic group of prime power order acting fixed-point-freely on $\fit_2(G)/\fit(G)$. In the second case, $\fit_2(G)/\fit(G)$ contains a unique element of order $2$ so it does not allow fixed-point-free extensions, i.e., $G=\fit_2(G)$. 

This leaves us with two more options for $\pi(G)$ according to Lemma \ref{lem:form_of_solvable_pi}.

	\item Case $\pi(G)=(\{p,q,r\},\{\{p,q\}\})$: If $\fit(G)\cong C_{pq}$ holds, then $G/\fit(G)$ acts as a fixed-point-free subgroup of $\aut(C_p)$ and $\aut(C_q)$ simultaneously, in
	particular, $G/\fit(G)$ is cyclic.
	
	Otherwise, $\fit(G)$ is either a $p$-group or a $q$-group. Without loss of generality, assume that $\fit(G)$ is a  $q$-group. Then there is an element of order $pq$ in $G$ and together with $\fit(G)$ it spans a semidirect product of a $q$-group with $C_p$, having non-empty prime graph. Thus, Lemma \ref{lem:semidirect_nontrivial_prime_graph} implies that $\fit(G)$ is trivial, a contradiction.
	
	We established that the only case left to consider is the case where $\fit(G)$ is an $r$-group. Then $\fit_2(G)/\fit(G)\cong C_p\times C_q$, otherwise we can repeat the argument from above, since $G/\fit(G)$ now has to contain elements of order $pq$. If it holds that $G=\fit_2(G)$, we are done. Otherwise $G/\fit_2(G)$ acts on $\fit_2(G)/\fit(G)\cong C_{pq}$ as a subgroup of $\aut(C_{pq})$, where the latter is an abelian group. Without loss of generality assume that $q<p$, so $p$ divides $|G|$ exactly once by Lemma~\ref{lem:sylows_nonisolated_elements}. If $G/\fit_2(G)$ contains elements of order $q$, then 
by Lemma~\ref{lem:centralizers_nonisolated_elements}, $G/\fit(G)$ contains a subgroup isomorphic to $C_q\times C_q$, acting on $\fit(G)$ without fixed-points. This is not possible, see~\cite[10.5.3]{robinson_1996}. Thus, $G/\fit_2(G)$ is an abelian $r$-group acting on $\fit_2(G)/\fit(G)$ as a fixed-point-subgroup of $\aut(C_p)$ and $\aut(C_q)$ simultaneously.

	\item {Case $\pi(G)=(\{p,q,r,s\},\{\{p,q\},\{r,s\}\})$:} If $\fit(G)\cong C_{rs}$ then $G/\fit(G)$ is abelian with at least two prime divisors $p$ and $q$, so 
	$G/\fit(G)\cong C_{pq}$ and this group must act on $\fit(G)$ without fixed points.
	Otherwise $\fit(G)$ must be of prime power order (since it is a nilpotent power-cograph group) and then $G$ cannot be a Frobenius-group, where $\pi(\fit(G))$ would form a connected component in $\pi(G)$. Thus, $G$ is $2$-Frobenius with corresponding normal subgroups $K\trianglelefteq N\trianglelefteq G$  (cf. Lemma~\ref{lem:disconnected_pi}). Then $\pi(N/K)$ forms a component in $\pi(G)$ which must be of size two by the assumptions on $\pi(G)$. Without loss of generality let $N/K\cong C_{rs}$ and let $K$ be a $q$-group. By definition, $G/K$ is a semidirect product with normal subgroup $N/K$ so we may choose some subgroup $H\leq G$ such that $K\leq H$ and $H/K$ is a complement of $N/K$ in $G/K$. If both $p$ and $q$ divide $|H/K|$, then $H/K\cong C_{pq}$, since $H/K$ acts on $N/K$ as an automorphism subgroup and $\aut(N/K)$ is abelian. In this case, $H\cong (K.C_q)\rtimes C_p$ with non-trivial prime graph, implying $K=\{1\}$ via Lemma~\ref{lem:semidirect_nontrivial_prime_graph}, a contradiction. Otherwise, $H/K$ is a $p$-group (because $p$ must divide $|H/K|$ at least once) and then $K\cong C_q$, again using Lemma~\ref{lem:semidirect_nontrivial_prime_graph} and the fact that some element of order $q$ must commute with an element of order $p$.
 But then $C_G(K)\cong C_p\times C_q$ is normal in $G$, contradicting the fact that $K=\fit(G)$ is a $q$-group.
\end{enumerate}
Finally, we point out that all groups appearing in the lemma's statement are indeed power-cograph groups according to Lemma~\ref{lem:characterization_via_centralizers}.
\end{proof}

While it would be even better to have a list of possible Fitting-subgroups in Cases 1) and 2b), the restrictions we give precisely describe the structure of $G$. It remains the question of which $p$-groups admit fixed-point-free automorphisms of prime order. For abelian Fitting-subgroups the question can be answered in full generality~\cite{mayr_1999}, so we obtain a classification under this extra assumption. For the general case, there is an extensive body of research concerned with bounding the Fitting-length of groups admitting fixed-point-free automorphisms of prime order, see for example~\cite{higman_1957, ward_1969}. In particular, they are always nilpotent by a result of Thompson~\cite{thompson_1959}, but no classification is available.

\section{Non-solvable groups}\label{section:non-solvable}
To deal with non-solvable groups we first analyze the socle of a power-cograph group to split up the classification task into smaller sub-problems.
Since a high number of direct factors in $\soc(G)$ implies the existence of relatively large element centralizers, the socle of a power-cograph turns out to be quite restricted.

\begin{lem}\label{lem: almost_simple_or_abelian_socle}
	Let $G$ be a non-solvable group whose power graph is a cograph. Set $S:=\soc(G)$. Then one of the following holds:
	\begin{enumerate}
		\item $S$ is a non-abelian simple group and $S\trianglelefteq G\leq Aut(S)$, so $G$ is almost simple.
		\item $S\cong C_p^n$ for some $n\in\mathbb{N}$ and $G/C\leq$GL$_n(\mathbb{F}_p)$, where
		$C:=C_G(S)$ is a $p$-group.
	\end{enumerate}
\end{lem} 
\begin{proof}
	We can write
\[
	S=A\times T_1\times\dots\times T_m
\]where $A$ is an abelian group and the $T_i$ are non-abelian simple groups. If $m\geq 1$ and $A$ is non-trivial, then there exists $a\in A$ such that, without loss of generality, $T_1\in C_{G}(a)$ holds. Since $|T_1|$ is divisible by at least three primes, the order of $C_G(a)$ is divisible by at least three primes but by Lemma~\ref{lem:centralizers_nonisolated_elements}, this is impossible. An analogous argument applies if $m>1$ holds, even if $A$ is trivial.

Thus $S$ is either an abelian group or a non-abelian simple group. In the second case, $G$ is almost simple. In the first case $A$ is nilpotent and so it is either a $p$-group or isomorphic to $C_p\times C_q$ for primes $p\neq q$. Assume first that $S\cong C_p\times C_q$ holds. Then 
$G/C_G(S)\leq\aut(S)$ and since the latter is abelian, $C_G(S)$ must have a non-abelian simple factor, contradicting Lemma \ref{lem:centralizers_nonisolated_elements}. Thus, $S\cong C_p^n$ for some prime $p$ and it follows that $C_G(S)$ is also a $p$-group. Otherwise there is a $q$-element centralizing $S$ for some prime $q\neq p$, forcing $n$ to be $1$ according to Theorem \ref{thm:nilpotent_poco}. In the same way as before, this contradicts Lemma \ref{lem:centralizers_nonisolated_elements} through $G$ being non-solvable.
\end{proof}

\subsection{The almost simple case}

Recall that the class of power-cograph groups is closed under taking subgroups.
	This means whenever $G$ is an almost simple power-cograph group, the socle $\soc(G)$ is a simple power-cograph group. We refer to Theorem~{\ref{thm:finite_simple_po-cographs}} for a characterization of non-abelian finite simple groups that are power-cograph groups.

\begin{lem}\label{lemma: almost_simple_to_simple_or_pgl}
	If $G$ is an almost simple power-cograph group, then $G$ is either simple or $G=\pgl_2(q)$ for some prime power $q$ or $G=M_{10}$.
\end{lem}
\begin{proof}
		Let $\soc(G)$ be a non-abelian simple group so that $G$ is an almost simple group. Note that $\soc(G)$ must be a power-cograph group due to subgroup-closedness.
	Due to Theorem~\ref{thm:finite_simple_po-cographs}, it suffices to check automorphisms of $\psl_3(4)$, $\psl_2(q)$ and $\sz(q)$.
	
	We can immediately exclude the first case via computations we performed in GAP~\cite{GAP4}. 

\paragraph{Extensions of $\psl_2(p^f)$ by field automorphisms, $p$ odd:}
	Suppose $G=\psl_2(p^f)\langle \alpha \rangle$ where $p$ is an odd prime and $\alpha$ is a field automorphism whose order divides $f$. Then $\alpha$ is a field automorphism of $\mathbb{F}_p^f$ which acts on $\psl_2(q)$ by acting on the respective matrix representatives in $\spl_2(q)$ entry-wise.

	\textbf{Case $f$ even:}
	Then $p^2-1$ divides $p^f-1$ and thus there exist elements of order $(p-1)(p+1)/2$ in $\psl_2(p^f)$. If $p\neq 3$, then $(p-1)(p+1)$ can neither be a product of two primes nor a prime power (since $p>3$ is prime) so $\psl_2(p^f)$ is not a power-cograph group by Theorem~\ref{thm:finite_simple_po-cographs}. If $p=3$ and $f>2$, then either $40=(3^4-1)/2$ divides $(p^f-1)/2$ or $3^{2g}-1$ divides $p^f-1$ for some odd number $g$. If the first case holds, then again $\psl_2(p^f)$ is not a power-cograph group. In the second case, $(3^g-1)(3^g+1)$ divides $p^f-1$. Then $(3^g+1)$ is divisible by $4$ since $g$ is odd, and $3^g-1$ is not a power of $2$, so $(p^f-1)/2$ is divisible by $4m$, where $m$ is an odd number. As in the first case, $\psl_2(p^f)$ is not a power-cograph group. In conclusion, either $f=2$ and $p=3$ or $f$ is odd.

In the case $f=2$ and $p=3$, it holds that $G\cong S_6$, and $G$ is not a power-cograph.

	\textbf{Case $f$ odd:}
	Since $\alpha$ fixes elements of the base field $\mathbb{F}_p$, in particular all matrices defined over $\mathbb{F}_p$ are fixed. Therefore, $\psl_2(p) \lesssim C_G(\alpha)$. Then if $p> 3$, the order of $C_G(\alpha)$ must be divisible by at least three primes as the group $\psl_2(p)$ is a nonabelian simple group, and by Lemma~\ref{lem:centralizers_nonisolated_elements} this is impossible. If $p=3$, then $\psl(2,3)\cong A_4$ is centralized by $\alpha$. If $C_G(\alpha)$ fulfills the necessary conditions in Lemma~\ref{lem:centralizers_nonisolated_elements}, then the order of $\alpha$ must be $2$, but we assumed $f$ to be odd here.

\paragraph{Extensions of $\psl_2(p^f)$ by diagonal-field automorphisms, $p$ odd:}
 Here $G=\psl_2(p^f)\langle \alpha \rangle$ where $p$ is an odd prime and $\alpha=\delta\sigma$ is a product of a diagonal automorphism $\delta$ and a non-trivial field automorphism $\sigma$.  If $|\sigma|$ is odd, then $G$ contains a subgroup $\psl(2,p^f)\langle \sigma' \rangle$ for some field automorphism $\sigma'$ of odd order, and the argument from above implies that $G$ is not a power-cograph group. So $|\sigma|$ is even. If $f$ is even, the arguments from above leave us with the case of $\psl(2,9)\langle \alpha \rangle$, where $\alpha$ has order $2$. Then $G\cong M_{10}$ which is an EPPO-group, and in particular, it has a power-cograph.

\paragraph{Extensions of $\psl_2(2^f)$ by field automorphisms}
    Suppose now $G=\psl_2(q)\langle \alpha \rangle$ where $q$ is a power of $2$. Then every outer automorphism is a field automorphism, and each field automorphism fixes at least a copy of $\psl_2(2)\cong S_3$. This can only happen if the field automorphism is of order $2$ and has no fixed field larger than $\mathbb{F}_2$, which implies that $q=4$. But $\psl_2(4)\cong\psl_2(5)$ has an automorphism group isomorphic to $\pgl_2(5)$.
   
\paragraph{Extensions of Suzuki groups by field automorphisms:}
    We use a similar argument for the Suzuki group case as in the $\psl_2(q)$ case. Let $G=\sz(2^{f})\langle \alpha \rangle$ with $\alpha$ a field automorphism whose order divides $f$, where $f$ is odd and $f\geq 3$. Then $\alpha$ fixes a subgroup isomorphic to $\sz(2)$ of order $20$, and again the result follows by Lemma~\ref{lem:centralizers_nonisolated_elements}.
    
    In conclusion, the only remaining cases are groups of the form $G=\pgl_2(q)$ for an appropriate prime power $q$.
\end{proof}

\begin{lem}\label{lem: psl-G-pgl_then_G_psl_or579}
	Let $\psl_2(q)\trianglelefteq G\leq\pgl_2(q)$ and assume $G$ is non-solvable. Then $G$ is a power-cograph group if and only if one of the following holds:
	\begin{enumerate}
		\item $q>3$ and $G=\psl_2(q)$
		\item $q$ is $5$, $7$, or $9$.
	\end{enumerate}
\end{lem}
\begin{proof}
	If $q$ is even then $\psl_2(q)=\pgl_2(q)$, thus assume that $q$ is odd.	Then $\psl_2(q)$ has index $2$ in $\pgl_2(q)$, so if $G$ is not equal to $\psl_2(q)$ then $G=\pgl_2(q)$ holds.
	In the latter case, $G$ has cyclic subgroups of order $q-1$ and $q+1$ and their orders must be prime powers or a product of two distinct primes for $\pow(G)$ to be a cograph.
	But one of $q-1$ and $q+1$ is divisible by $4$, so one of them is a $2$-power and the other is $2r$ with a prime $r>2$.
	
	We could be dealing with the unique solution to the Catalan Conjecture, i.e., $q=9$ and $q-1=8$. Since $q=9$ gives $q+1=10$, $\pgl_2(9)$ is indeed power-cograph. Furthermore, $q>3$ implies $2<q-1,q+1$, so this is actually the only possible case if $q$ is not prime. 
	
 	We are left with two cases, since $q$ is a prime of the form $2^a\pm 1$ now.
 	If $q=2^a+1$ then $q+1=2(2^{a-1}+1)$, so $2^{a-1}+1$ must be prime as well. A well-known result (see for example~\cite[Page 1]{lectures_on_Fermat_numbers}) states that both $a$ and $a-1$ must be powers of $2$ (including the possibility $1=2^0$) which is only possible if $a=2$. This gives us $\pgl_2(5)$, which is again power-cograph. If $q=2^a-1$ then $q-1=2(2^{a-1}-1)$ and so $2^{a-1}-1$ is an odd prime. It is well-known that then both $a$ and $a-1$ have to be primes (see for example~\cite[Appendix B, Theorem B.1]{lectures_on_Fermat_numbers}), so $a=3$. Indeed $\pgl_2(7)$ is also a power-cograph group.
\end{proof}

\subsection{The abelian socle case}

Recall that $2$-Frobenius groups are always solvable. We observe the following additional restriction for non-solvable groups.
\begin{lem}
If $G$ is a non-solvable power-cograph group, then $G$ is not a Frobenius group.
\end{lem}
\begin{proof}
Let $G = K \rtimes H$ be a non-solvable Frobenius group. Then the Frobenius complement $H$ (and thus $G$) contains a subgroup isomorphic to $\mathrm{SL}_2(5)$, which has non-trivial center and more then $3$ primes dividing its order. In particular, $\spl_2(5)$ is not a power-cograph group according to Lemma~\ref{lem:centralizers_nonisolated_elements}. 
\end{proof}

\begin{lem}\label{lem:2_isolated}
If $G$ is a non-solvable power-cograph group with an abelian socle, then $2$ is isolated in $\pi(G)$. In particular, $\pi(G)$ is disconnected.
\end{lem}

\begin{proof}
Let $S=\soc(G)$. If $S=C_p$, then consider the action of $G$ on $S$ via conjugation. This induces $\phi:G\to \aut(S)$ with kernel $C_G(S)$. Since $G$ is non-solvable, either $\ker\phi$ or $\text{Im }\phi$ contains a non-abelian simple composition factor. However, $\aut(S)$ is abelian and $C_G(S)$ cannot contain this factor by Lemma \ref{lem:centralizers_nonisolated_elements}, so this is a contradiction. 

Now either $\pi(G)$ is disconnected and then all prime divisors of $|\soc(G)|$ must be in the component of $2$ (Lemma~\ref{lem:disconnected_pi}) or $\pi(G)$ forms a star with center $2$ (Lemma~\ref{lem:pi_star}). In any case, odd primes divide $|\soc(G)|$ at most once and since we can rule out cyclic socles by non-solvability of $G$, we can assume that 
$S=\mathbb{F}_2^d$ with $d\geq 1$. Assume that $2$ is not isolated in $\pi(G)$, so there exists $g\in G$ such that $|g|=p\neq 2$ with $|C_G(g)|$ even. Suppose $w\in C_G(g)$ such that $w$ has order $2$.

Let $H$ be a Sylow $2$-subgroup of $G$ with $w\in H$.
Since $S\trianglelefteq G$ and $S$ is a $2$-group, we have $S\leq O_2(G)$, where $O_2(G)$ denotes the largest normal $2$-subgroup of $G$. Note that $O_2(G)$ is the intersection of all Sylow $2$-subgroups of $G$.  Then $S\trianglelefteq H$ and $S\cap Z(H)$ is non-trivial because $H$ is a Sylow $2$-subgroup and any normal subgroup intersects the center non-trivially. So there exists $z\in S\cap Z(H)$ such that $|z|=2$ and $z\in C_G(w)$. Now suppose $z \not \in C_G(g)$. Since $z$ centralizes $w$ and $w$ centralizes $g$, we have that $z$ normalizes $\langle g \rangle \trianglelefteq C_G(w)$. Furthermore, $zgz^{-1}=zgz=g^{-1}$ because $z$ commutes with $w$ and it has order $2$. Now $g$ normalizes $S$ so $gzg^{-1}\in S\setminus \langle z \rangle$ because $g$ and $z$ do not commute. Thus $\langle z, gzg^{-1} \rangle \cong V_4$ since it is a subgroup of $\mathbb{F}_2^d$ and it is not cyclic. However, $zgzg^{-1}=g^{-2}$ is of order $p$ and an element of $V_4$, which is a contradiction.
Thus $z\in C_G(g)$ and $g\in C_G(z)$. By Lemma~\ref{lem:centralizers_nonisolated_elements}, $C_G(z)=\langle g \rangle \rtimes Q$, where $Q$ is a $q$-group of rank at most $2$. As $S\subseteq C_G(z)$, $S\cong \mathbb{F}_2$ or $S\cong \mathbb{F}_2^2$. Then $G/ C_G(S)\leq \aut(S)$ (as above, by Lemma~\ref{lem:centralizers_nonisolated_elements}, $C_G(S)$ cannot have non-abelian simple composition factors) must have a non-abelian simple composition factor, which is impossible due to $\mathrm{GL}_2(2)\cong S_3$.
\end{proof}

\begin{lem}
If $G$ is a non-solvable power-cograph group, then $G$ has a normal subgroup $N$ that is either trivial or a $2$-group, such that in any case $G/N$ is almost simple. Moreover, $G/N$ is a power-cograph group.
\end{lem}
\begin{proof}
If $N=1$, then the statement follows from the definition, so assume otherwise.
Combining Lemma~\ref{lem:disconnected_pi} and Lemma~\ref{lem:2_isolated}, we may assume that $G$ is an extension of a nilpotent $\pi_1$-group by a simple group by a $\pi_1$-group, where $\pi_1$ denotes the component containing $2$ in $\pi(G)$. 
We showed earlier that $2$ is isolated in $\pi(G)$ if $\soc(G)$ is a $2$-group, and so $G/\mathcal{O}_2(G)\cong T.H$, where $H$ is a $2$-group and $T$ is simple.

To see that $\pow(G/N)$ is a cograph, recall Lemma \ref{lem:lifting_P4}. Since $2$ is isolated in $\pi(G)$, any induced $P_4$ in $\pow(G/N)$ would contain odd-order elements only and also the relevant centralizer defined in Lemma \ref{lem:lifting_P4} must be of odd order. Thus, all induced $P_4$'s in $\pow(G/N)$ would lift to $G$.
\end{proof}

We refer to Theorem~\ref{thm:finite_simple_po-cographs} once again to deduce that $T$ 
may only be one of $\psl_3(4)$, $\psl_2(q)$ where $q$ is a prime power, or $\sz(q)$ where $q$ is a power of $2$.

\begin{lem}\label{lem: abelian-socle_quotient_by_O2}
	Let $G$ be a non-solvable power-cograph group with abelian socle, then $S:=\soc(G)\cong\mathbb{F}_2^d$ with $d\geq 1$.
	Moreover, $G/\mathcal{O}_2(G)$ has a non-abelian simple socle $T$ and one of the following holds:
	\begin{enumerate}
		\item $T\cong\psl_2(2^n)$ with $n\geq 2$. Then $S=\mathcal{O}_2(G)$ and each minimal normal subgroup of $G$ is isomorphic to the natural $\mathbb{F}_{2^n}[\textsc{SL}_2(2^n)]$-module as a $G/S$-module.
		\item $T\cong\sz(2^{2e+1})$ with $e\geq 2$. Then $S=\mathcal{O}_2(G)$ and each minimal normal subgroup of $G$ is isomorphic to the natural $[\mathbb{F}_{2^{2e+1}}\sz(2^{2e+1})]$-module as a $G/S$-module.
	\end{enumerate}
	In both cases $G$ is a power-cograph group if and only if $T$ is a simple power-cograph group where $2$ is isolated in $\pi(G)$.
\end{lem}
\begin{proof}
		By the previous lemma and Lemma~\ref{lem: almost_simple_or_abelian_socle}, $G/\mathcal{O}_2(G)$ has a non-abelian simple socle whose power graph is a cograph. According to \cite[Theorem 1.3]{cameron2021finite}, $T:=\soc(G/\mathcal{O}_2(G))$ is either $\psl_3(4)$, $\psl_2(q)$ or $\sz(q)$ for appropriate values of $q$. We may further assume that $2$ is isolated in $\pi(G)$ by Lemma \ref{lem:2_isolated} and in particular, all odd order elements in $G$ must act on $S$ fixed-point-freely.
		
		We checked with GAP~\cite{GAP4} that we can exclude $\psl_3(4)$ by explicitly computing odd order elements with fixed-points in all irreducible representations of $\psl_3(4)$ over $\mathbb{F}_2$.	
		
		We can exclude $\psl_2(q)$ for $q>5$ and odd via~\cite[Proposition 3.2]{stewart_1973}. Recall that $\psl_2(5)\cong\psl_2(4)$, so this leaves us with $\psl_2(q)$ and $\sz(q)$ for $q$ even.

		Furthermore, by~\cite[Proposition 4.1]{stewart_1973} and~\cite[Theorem, Remark 1]{martineau_1972}, 
		it must be the case that $\mathcal{O}_2(G)$ is elementary abelian and forms a direct sum of copies of the natural SL$_2(q)$-module or the natural $\sz(q)$-module, respectively.
		
		We now prove that under these assumptions, $H:=S.T\leq G$ is a power-cograph group whenever $T$ is. The general approach is to apply Lemma~\ref{lem:characterization_via_centralizers}.
		
		Consider commuting elements of distinct prime orders in $H$, say $x$ and $y$ . Then 
		$C_H(x)S/S$ is either a $p$-group or isomorphic to $C_p\times C_q$ for appropriate primes $p$ and $q$, since this is the case for $T$ whenever $T$ is a power-cograph group, see Theorem~\ref{thm:finite_simple_po-cographs}. In the first case, one of $x$ and $y$ has order $2$ and since $2$ is isolated in $T$ the respective element, $y$ say, must be contained in $S$.

		In the second case, $C_H(x)$ is either $C_p\times C_q$ and thus admissible according to Lemma \ref{lem:characterization_via_centralizers}, or there is some non-trivial element in $S$ that commutes with an element of order $p\neq 2$. In any case, the only obstruction to $C_H(x)$ being admissible is the existence of an element $t\in T$ of prime order $p\neq 2$, such that $C_S(t)\neq\{1\}$ holds. We argue that such elements cannot exist.
		
		Let $q:=2^n$ and let $A\in\textsc{SL}_2(q)$. If $A$ has eigenvalue $1$, then either $A$ is diagonalizable and $\det(A)=1$ forces $A$ to be the identity matrix, or the minimal polynomial of $A$ is $(X-1)^2\equiv_2 X^2-1$ and $A$ is of order $2$.	
		
		If $A\in\sz(q)$, in the natural representation $A$ is a $4$-by-$4$ matrix over $\mathbb{F}_q$. Odd prime divisors of $|\sz(q)|$ divide $q-1$ or $q^2+1$.
		If $A$ has an order dividing $q^2+1$ then $A$ cannot fix any non-zero vector in the natural representation, or otherwise the order of $A$ would also divide $q^3|\textsc{GL}_3(q)|=q^6(q-1)^2(q^2-1)(q^2+q+1)$ (this follows since with respect to a suitable basis $A$ would be a product of a block-diagonal matrix with blocks of size 1 and 3 with an upper triangular matrix). But $q^2+1$ is coprime to both $q^2-1$ and $q^2+q+1$ if $q$ is even.	Other elements of odd prime order must have their order divide $q-1$. In the natural representation, these elements are either diagonal 
(and only ever fix a non-zero vector if they are the identity), or they are elements of order $5$ coming from copies of $\sz(2)$ in $\sz(q)$. However, these have irreducible minimal polynomials of degree $4$ over $\mathbb{F}_2$ and in particular they do not admit eigenvalue $1$  over any extension field of $\mathbb{F}_2$.

Finally, note that $G/\mathcal{O}_2(G)$ cannot be a non-simple almost simple group, since otherwise by Lemma~\ref{lem: psl-G-pgl_then_G_psl_or579}, $G/\mathcal{O}_2(G)$ would have to be isomorphic to $\pgl_2(q)$ with $q\in\{5,7,9\}$ and then $2$ is not isolated in $\pi(G)$.
\end{proof}

\section{Proofs of the main theorems}\label{section:summary}

We conclude this work by synthesising our results into proofs of Theorem~\ref{mainthm2} and Theorem ~\ref{mainthm1}.
\begin{proof}[The proof of Theorem~\ref{mainthm2}]
Let $G$ be a finite non-solvable group. Suppose that $G$ is a power-cograph group. We reduce the case to almost simple groups or groups with an abelian socle in Lemma~\ref{lem: almost_simple_or_abelian_socle}. We consider the cases separately. In Lemma~\ref{lemma: almost_simple_to_simple_or_pgl} we proved that in the almost simple case either $G$ is itself simple or isomorphic to $\pgl_2(q)$ where $q$ is a prime power or isomorphic to $M_{10}$. If $G$ is a simple group,  then~\cite[Theorem 1.3]{cameron2021finite} holds. Further, we showed in Lemma~\ref{lem: psl-G-pgl_then_G_psl_or579} that if $G$ lies between $\psl_2(q)$ and $\pgl_2(q)$  then either $G$ is $\psl_2(q)$ or one of $\pgl_2(5)$, $\pgl_2(7)$ or $\pgl_2(9)$. This settles parts (1) and (2) of Theorem~\ref{mainthm2}. In the abelian socle case, we proved that $\pi(G)$ is always disconnected in Lemma~\ref{lem:2_isolated}. Finally, in Lemma~\ref{lem: abelian-socle_quotient_by_O2} we established Parts (3) and (4) of Theorem~\ref{mainthm2} by showing that the quotient space $G/\soc(G)$ is isomorphic to either $\psl_2(2^n)$ with $n\geq 2$ or $\sz(2^{2e+1})$ with $e\geq 2$ and is itself a power-cograph group. For the other direction, Lemma~\ref{lem: psl-G-pgl_then_G_psl_or579} and Lemma~\ref{lem: abelian-socle_quotient_by_O2} also show that if $G$ is one of the groups in Parts (1--4) then $G$ is a power-cograph group.
\end{proof}

\begin{proof}[The proof of Theorem~\ref{mainthm1}]
	Let $G$ be a finite solvable group which is a power-cograph group. If $\pi(G)$ is connected, then it is a $P_3$ by Lemma~\ref{lem:form_of_solvable_pi}, and Lemma~\ref{lem:classification_connected_pi} proves that one of Parts (1) or (2) of Theorem ~\ref{mainthm1} holds. Otherwise, if $\pi(G)$ is disconnected, then by Lemma~\ref{lem:form_of_solvable_pi} it is an induced subgraph of $2K_2$ and Lemma~\ref{lem: classification_disconnected_pi} shows the rest of Theorem~\ref{mainthm1} holds. 
	For the other direction, let $G$ be one of the groups in Parts (1--5). Then $G$ is a power-cograph group as shown in Lemma~\ref{lem:classification_connected_pi} and Lemma~\ref{lem: classification_disconnected_pi}.
\end{proof}

\section{Acknowledgements}

We thank an anonymous referee for pointing out a gap in a previous proof of Lemma~\ref{lemma: almost_simple_to_simple_or_pgl} which led to the inclusion of the group $M_{10}$ into Theorem~\ref{mainthm2}.

\bibliographystyle{plain}
\bibliography{refs2}

\end{document}